\documentclass[12pt, oneside]{amsart}   	
\allowdisplaybreaks
\usepackage{hyperref}
\usepackage{geometry,comment}                		
\usepackage{graphicx}				
\usepackage{amssymb}

\newtheorem{theorem}{\bf Theorem}[section]
\newtheorem{proposition}[theorem]{\bf Proposition}
\newtheorem{lemma}[theorem]{\bf Lemma}

\newtheorem{definition}[theorem]{\bf Definition}

\newtheorem{question}[theorem]{\bf Question}

\newenvironment{proofof}[1]{\noindent{\it Proof of
#1.}}{\hfill$\square$\\\mbox{}}

\def\field{K}
\def\sepbeta{\beta_{\mathrm{sep}}}

\title{Separating versus ordinary Noether numbers} 
\author[M. Domokos]{M\'aty\'as Domokos}
\address{HUN-REN Alfr\'ed R\'enyi Institute of Mathematics,
Re\'altanoda utca 13-15, 1053 Budapest, Hungary,
ORCID iD: https://orcid.org/0000-0002-0189-8831}
\email{domokos.matyas@renyi.hu}
\author[B. Schefler]{Barna Schefler} 
\address{E\"otv\"os Lor\'and University, 
P\'azm\'any P\'eter s\'et\'any 1/C, 1117 Budapest, Hungary} 
\email{scheflerbarna@yahoo.com}
\subjclass[2020]{Primary 13A50; Secondary 13P15, 20C15}
\keywords{polynomial invariants, separating sets, degree bounds, finite groups, Noether number}

\begin{document}
\maketitle

\begin{abstract} 
Let $G$ be a finite group and $\field$ a field containing an element of multiplicative order $|G|$. 
It is shown that if $G$ has a cyclic subgroup of index at most $2$, 
then the separating Noether number over $\field$ of $G$ coincides with the Noether number 
over $\field$ of $G$. 
The same conclusion holds when $G$ is the direct product of a dihedral group and the $2$-element group. 
On the other hand, the smallest non-abelian groups $G$ are found for which the separating 
Noether number over $\field$ is strictly less than the Noether number over $\field$. Along the way 
the exact value of the separating Noether number is determined for 
all groups of order at most $16$. The results show in particular that unlike the ordinary Noether number, the separating 
Noether number of a non-abelian finite group may well be equal to the separating Noether number of a proper direct  factor of the group. 
\end{abstract}

\section{Introduction}
The separating Noether number over a field $\field$ of a finite group $G$ is 
the minimal positive integer $d$ such that for any finite dimensional $\field$-linear representation of $G$, 
any two 
dictinct $G$-orbits can be distinguished by a polynomial 
$G$-invariant of degree at most $d$.  
For a finite dimensional vector space $V$ over $\field$ we write $\field[V]$ for the 
coordinate ring of $V$ 
identified with the $n$-variable polynomial algebra 
$\field[x_1,\dots,x_n]$, where the variables $x_1,\dots,x_n$ represent the coordinate functions on $V$ with respect to a chosen basis. 
Let $G$ be a finite group, and $\rho:G\to \mathrm{GL}(V)$ a representation of $G$ on $V$. We shall say that \emph{$(V,\rho)$ is a $\field G$-module} in this case; 
mostly we shall suppress $\rho$ from the notation and we will say that 
\emph{$V$ is a $\field G$-module}. 
Sometimes we write $V_\field$ instead of $V$ when we want to make explicit the actual base field. 
The action of $G$ on $V$ induces an action on $\field[V]$ via $\field$-algebra automorphisms,
such that the variables $x_1,\dots,x_n$ span a $\field G$-submodule isomorphic to $V^*$ (the $G$-module dual to $V$), 
and $\field[V]$ is the symmetric tensor algebra of $V^*$.   
Consider 
\[\field[V]^G=\{f\in \field[V]\mid \forall g\in G:\quad g\cdot f=f\},\] 
the \emph{subalgebra of 
$G$-invariants in  $\field[V]$}. 
For $v\in V$ we shall denote by $G\cdot v$ the $G$-orbit of $v$, 
and by $\mathrm{Stab}_G(v):=\{g\in G\mid gv=v\}$ 
the \emph{stabilizer subgroup in $G$ of $v$}.

It is well known that for $v,w\in V$ with $G\cdot v\neq G\cdot w$, 
there exists an invariant $f\in \field[V]^G$ with $f(v)\neq f(w)$ (see for example \cite[Theorem 3.12.1]{derksen-kemper}, 
\cite[Theorem 16]{kemper} or \cite[Lemma 3.1]{draisma-kemper-wehlau}).  
Therefore the general notion of a "separating set of invariants"   
(see \cite[Definition 2.4.1]{derksen-kemper}) reduces to the following in the case of finite groups: 

\begin{definition}\label{def:separating set}
A subset $S$ of $\field[V]^G$ is called a \emph{separating set} if for any $v,w\in V$, 
$f(v)=f(w)$ for all $f\in S$ implies that $v$ and $w$ have the same $G$-orbit. 
\end{definition} 

The result mentioned before Definition~\ref{def:separating set} implies that every generating set of the algebra $\field[V]^G$ is a separating set. 
In the structural study of algebras of polynomial invariants of finite groups the focus on separating sets (instead of the more special generating sets) turned out to be useful in finding strengthened versions of known theorems in the invariant theory of finite groups, see for example \cite{dufresne}, 
\cite{dufresne-elmer-kohls}. Considering explicit computations of invariants, the class of modular representations for which generators of the corresponding algebra of invariants is known is very limited. On the other hand, in several such cases explicit finite separating sets are known, see for example 
\cite{sezer}, \cite{kohls-sezer}. 
The minimal possible size of a separating set is investigated in 
\cite{dufresne-jeffries}. 
In the non-modular case, an explicit separating set of multisymmetric polynomials is given in \cite{lopatin-reimers}. 
The systematic study of separating invariants over finite fields began recently in \cite{kemper-lopatin-reimers}, 
\cite{lopatin-muniz}. 

\subsection{Degree bounds} 
Endow $\field[V]$ with the standard grading, so the variables $x_1,\dots,x_n$ span the degree $1$ homogeneous component of $\field[V]$. We shall denote by $\deg(f)$ the degree of a 
non-zero homogeneous element of $\field[V]$. 
Obviously, $\field[V]^G$ is spanned by homogeneous elements, so it inherits the grading from 
$\field[V]=\bigoplus_{d=0}^\infty\field[V]_d$. We shall write $\field[V]^G_{\le d}$ for the sum of the homogeneous components of degree at most $d$. Degree bounds for generators of algebras of invariants have always been in the focus of invariant theory 
(see e.g. \cite{wehlau}). 
We set
\[\beta(G,V):=\min\{d\in \mathbb{N}_0\mid \field[V]^G_{\le d}\text{ generates the }\field\text{-algebra }\field[V]^G\},\] 
\[\beta^\field(G):=\sup\{\beta(G,V)\mid V\text{ is a finite dimensional }\field G\text{-module}\}.\] 
The quantity $\beta^\field(G)$ is called the \emph{Noether number of $G$} (over $\field$) because 
of the classical inequality $\beta^{\mathbb{C}}(G)\le |G|$ due to E. Noether \cite{noether} (the inequality was extended to the non-modular case in \cite{fleischmann}, \cite{fogarty}).  
Its systematic study was initiated in \cite{schmid}. The results of \cite{schmid} 
were improved first in \cite{domokos-hegedus}, \cite{sezer:1} 
and later in 
\cite{CzD:1}, \cite{cziszter-domokos:indextwo}. In particular, 
\cite[Table 1]{cziszter-domokos-szollosi} gives the exact value of the Noether number for all non-abelian groups of order less than $32$, and \cite{cziszter-domokos:indextwo} gives the Noether number for all groups containing a cyclic subgroup of index two. 

Following \cite[Definition 2]{kohls-kraft} (see also \cite{kemper}) we set 
\[\sepbeta(G,V):=\min\{d\in \mathbb{N}_0\mid \field[V]^G_{\le d}\text{ is a separating set}\},\] 
and 
\[\sepbeta^\field(G):=\sup\{\sepbeta(G,V)\mid V\text{ is a finite dimensional }G\text{-module}
\text{ over }\field\}.\] 
We shall refer to the above quantities as the \emph{separating Noether number of $G$ over $\field$}. 
As different orbits of a finite group can be separated by polynomial invariants, we have 
the obvious inequality 
$\sepbeta(G,V)\le \beta(G,V)$ for any $V$, 
and hence 
\begin{equation}\label{eq:sepbeta<beta}
\sepbeta^\field(G)\le \beta^\field(G).
\end{equation}
A major difference between the Noether number and its separating version is that if the characteristic of the base field  $\field$ divides the order of $G$, then $\beta^\field(G)=\infty$ by \cite{richman}, whereas $\sepbeta^\field(G)\le |G|$ also in the modular case 
(see for example the proof of \cite[Theorem 3.12.1]{derksen-kemper}, \cite[Theorem 16]{kemper} or \cite[Lemma 3.1]{draisma-kemper-wehlau}).  For some results on degree bounds on separating invariants in 
the modular case see \cite[Theorem C]{kohls-kraft}, \cite{elmer-kohls}.
However, in the present paper our focus is on the non-modular setup, when the characteristic of the base field does not divide the order of $G$. Interest in these degree bounds (over the real or complex field) 
arose also in recent applications for data science and machine learning, see 
\cite{bandeire-blumsmith-kileel-niles-weed-perry-wein}, \cite{cahill-contreras-hip}, \cite{cahill-iverson-mixon}, 
\cite{cahill-iverson-mixon-packer}. We mention also that problems in signal processing motivated the study of a variant of the Noether number for rational invariants in 
\cite{blumsmith-garcia-hidalgo-rodriguez} (see \cite{blumsmith-derksen} for a very recent result on this). 

In this non-modular context an interesting and generally open question is the following: 
\begin{question} 
To what extent is the change from "generating sets" to "separating sets" reflected in the corresponding degree bounds? 
\end{question}

It is proved in \cite[Corollary 3.11]{domokos:abelian} that for an abelian group $G$ and a field $\field$ 
containing an element of multiplicative order $\exp(G)$, typically 
$\sepbeta^\field(G)$ is strictly smaller than $\beta^\field(G)$. 
It is natural to expect similar typical behavior for non-abelian finite groups as well, 
but as far as we know, the only information in the literature concerning this question 
is in \cite{cziszter:C7rtimesC3}, where the separating Noether 
number of the non-abelian groups of order $3p$
for some prime $p$ congruent to $1$ modulo $3$ is determined. 
One of the obstacles is that the exact value of the Noether number is known only 
for a limited class of groups. This is the case already for abelian groups, whose 
Noether number equals their Davenport constant.   
However, the separating Noether number was determined recently for classes of 
abelian groups whose Noether number (Davenport constant) is not known, see 
\cite{schefler_c_n^r}, \cite{schefler-zhao-zhong}.  
This may be a hint that perhaps the separating Noether number can be computed 
also for groups whose Noether number is out of reach. As a natural start one should first 
compute the separating Noether number at least for those groups whose Noether number is known. 
This is done in \cite{schefler_rank2}, \cite{schefler-zhao-zhong} 
in the case of abelian groups,  
and we begin this project for non-abelian groups in the present paper. 
The Noether numbers of all the groups with a cyclic subgroup of index at most $2$ 
and for the remaining groups of order less than $32$
were computed in \cite{cziszter-domokos:indextwo} and 
\cite{cziszter-domokos-szollosi}, respectively. 

Here we determine the separating Noether numbers for the groups 
with a cyclic subgroup of index at most $2$ (as well as for the direct product 
of a dihedral group and the $2$-element group), 
and for all groups of order at most $16$, working  
over a base field $\field$ having an element whose multiplicative order 
equals the order of the group 
(note that it turns out that 
the smallest non-abelian groups $G$ with $\sepbeta(G)<\beta(G)$ 
have order $16$). 
Our arguments do not rely on computer calculations (see Section~\ref{sec:algorithm} 
for comments on possible use of computers and its theoretical limitations in computation of the separating Noether number for concrete groups). 
The present work is a shortened version of our preprint \cite{domokos-schefler}, whose other results (covering the remaining 
groups of order less than $32$) will be published elsewhere.

{\bf{Running assumption.}}
Throughout this paper $G$ will stand for a finite group, and 
$\field$ for a field. 
Unless explicitly stated otherwise, it is assumed  
that the characteristic of $\field$ does not divide $|G|$. 

\section{Main results}\label{sec:main results}

Our first result builds on \cite{cziszter-domokos:indextwo} (see the equality \eqref{eq:beta(index two cyclic)}) 
and shows that for a group with a cyclic subgroup of index $2$ (their classification will be recalled in 
Section~\ref{sec:indextwo})
the inequality \eqref{eq:sepbeta<beta} holds with equality, namely: 
\begin{theorem}~\label{thm:sepbeta index two} 
Let $G$ be a non-cyclic finite group with a cyclic subgroup of index $2$,  
and 
assume that $\field$ contains an element of multiplicative order $|G|$. 
Then we have the equality 
\[\sepbeta^\field(G)=\frac{1}{2} |G|  +
\begin{cases}
2 	& \text{ if } G=\mathrm{Dic}_{4m}, \text{  $m>1$};\\
1	& \text{ otherwise. }
\end{cases}\]
\end{theorem} 

Using \cite[Corollary 5.5]{cziszter-domokos:indextwo} 
we are able to determine the separating Noether number for 
another infinite sequence of non-abelian finite groups: 

\begin{theorem}\label{thm:sepbeta(D2nxC2)}
For $n\ge 2$ even consider the direct product $G:=\mathrm{D}_{2n}\times \mathrm{C}_2$ of the 
dihedral group of order $2n$ and the cyclic group $\mathrm{C}_2$ of order $2$. 
Assume that $\field$ has an element with multiplicative order $n$. Then 
we have 
\[\sepbeta^\field(\mathrm{D}_{2n}\times \mathrm{C}_2)=n+2.\]  
\end{theorem}

There are $16$ non-abelian groups of order 
less than $32$ not covered by Theorem~\ref{thm:sepbeta index two} or 
Theorem~\ref{thm:sepbeta(D2nxC2)}. We discuss them  
in the following theorem: 

\begin{theorem}\label{thm:sepbeta(<32)}  
Let $G$ be a group of order at most $16$
that does not contain a cyclic subgroup of index at most $2$, and $G$ is not isomorphic to 
$\mathrm{D}_{2n}\times \mathrm{C}_2$. 
Assume that $\field$ contains an element of multiplicative order $|G|$. 
The value of $\sepbeta^\field(G)$ is given in the 
following table:  
\[ 
\begin{array}{c|c|c|c|c}
\mathrm{GAP} &G   & \beta^\field(G) & \sepbeta^\field(G) &\text{reference for }\sepbeta^\field(G) \\ \hline 
 (8,5)  &  \mathrm{C}_2\times \mathrm{C}_2\times \mathrm{C}_2 & 4 & 4 & \text{\cite[Theorem 3.10]{domokos:abelian}} 
 \\   
 (9,2)  &\mathrm{C}_3\times \mathrm{C}_3 & 5 &  4 & 
 \text{\cite[Theorem 1.2]{schefler_c_n^r}}
 \\ 
(12,3) &\mathrm{A}_4  & 6 & 6 & 
\mathrm{Theorem~\ref{thm:sepbeta(A4)}} \\
(16,2) & \mathrm{C}_4\times \mathrm{C}_4 & 7 & 6 & \text{\cite[Theorem 1.2]{schefler_c_n^r}}
  \\
(16,3) & (\mathrm{C}_2\times \mathrm{C}_2) \rtimes \mathrm{C}_4 =  (\mathrm{C}_4 \times \mathrm{C}_2) \rtimes_{\psi} \mathrm{C}_2  & 6 & 6 & 
\mathrm{Theorem~\ref{thm:sepbeta((C2xC2)rtimesC4)}}\\
(16,4) & \mathrm{C}_4 \rtimes \mathrm{C}_4 & 7 & 6 & 
\mathrm{Theorem~\ref{thm:sepbeta(C4rtimesC4)}}\\
(16,10) & \mathrm{C}_2\times \mathrm{C}_2\times \mathrm{C}_4 & 6 & 6 & \text{\cite[Theorem 3.10]{domokos:abelian}} \\
(16,12) & \mathrm{Dic}_8 \times \mathrm{C}_2  & 7 & 6 & 
\mathrm{Theorem~\ref{thm:sepbeta(Dic8xC2)}}\\
(16,13) & (Pauli) \; = \; (\mathrm{C}_4 \times \mathrm{C}_2) \rtimes_{\phi} \mathrm{C}_2   & 7  & 7 & 
\mathrm{Theorem~\ref{thm:sepbeta(Pauli)}}\\
 (16,14)  & \mathrm{C}_2\times \mathrm{C}_2\times \mathrm{C}_2\times \mathrm{C}_2 & 5 & 5 & \text{\cite[Theorem 3.10]{domokos:abelian}}
 \\
\hline
\end{array}
\] 
\end{theorem}
The first column in the above table contains the identifier of $G$ in the Small Groups Library of GAP (see \cite{GAP4}); that is, a pair of the form \texttt{(order, i)}, where the GAP command $\mathsf{SmallGroup}(\texttt{order,i})$  returns the $\texttt{i}$-th group of order $\texttt{order}$ in the catalogue. The second column gives $G$ in a standard notation or indicates its structure as a direct product or semidirect product 
(the symbol $\rtimes$ always stands for a semidirect product that is not a direct product, 
and $\psi$ and $\phi$ in the rows $(16,3)$ and $(16,13)$ stand for two different involutive 
automorphisms of $\mathrm{C}_4\times \mathrm{C}_2$). The presentation of 
$G$ in terms of generators and relations can be found in the beginning of the section 
that contains the theorem referred to in the last column for the value of $\sepbeta^\field(G)$. 
The third column of the table 
contains $\beta^\field(G)$; these values are taken from 
\cite{cziszter-domokos-szollosi}. The fourth column contains 
$\sepbeta^\field(G)$.  

Note that in  some cases the value for $\sepbeta^\field(G)$ given in the table 
holds under weaker restrictions on $\field$ than the restrictions stated in 
Theorem~\ref{thm:sepbeta(<32)}. 
The exact conditions on $\field$ needed 
for the validity of our arguments are contained in the statement 
quoted in the last column of the table.  

\subsection{Comments on the base field}\label{sec:dependence on base field}
We would like to stress that the separating Noether number is more sensitive 
for the choice of the base field than the Noether number. 
It is shown in \cite[Corollary 4.2]{knop} that the Noether number of a finite group 
depends only on the characteristic of the base field, so 
\[\beta^\field(G)=\begin{cases} \beta^{\mathbb{F}_p}(G) \text{ if } \mathrm{char}(\field)=p>0; 
\\ \beta^\mathbb{C}(G)  \text{ if } \mathrm{char}(\field)=0. 
\end{cases}\]
Moreover,  by \cite[Theorem 4.7]{knop}, for all but finitely many primes $p$ we have 
$\beta^{\mathbb{F}_p}(G)\ge \beta^\mathbb{C}(G)$. 

On the other hand, the behavior of the separating Noether number in this respect is very different. 
For example, it is proved in \cite{domokos:rational} that for any prime order group 
$\mathrm{C}_p$ we have $\sepbeta^\mathbb{Q}(\mathrm{C}_p)=3$ (except for 
$p=2$, when $\sepbeta^\mathbb{Q}(\mathrm{C}_2)=2$), 
whereas it is well known and easy to see that $\sepbeta^{\mathbb{C}}(\mathrm{C}_p)=p$ 
(and also $\sepbeta^{\mathbb{R}}(\mathrm{C}_p)=p$ by 
\cite[Proposition 3.13]{blumsmith-garcia-hidalgo-rodriguez}).   

However, let us record some general facts relating the dependence 
of the separating Noether number on the base field. 
 
\begin{lemma}\label{lemma:spanning invariants} 
Let $F$ be a subfield of a field $L$, and let $V_F$ be an $F G$-module.  
\begin{itemize}
\item[(i)] For any non-negative integer $d$, the $L$-vector space   
$L[V_L]^G_d$ is spanned by its subset $F[V_F]^G_d$. 
\item[(ii)] We have the inequality 
$\sepbeta(G,V_F)\le \sepbeta(G,V_L)$. 
\item[(iii)] We have the inequality 
$\sepbeta^F(G)\le \sepbeta^L(G)$. 
\end{itemize} 
\end{lemma}
\begin{proof} 
(i) It is well known, and follows from the fact that the space of solutions over $L$ of 
a system of homogeneous linear equations with coefficients in $F$ is spanned by the 
subset of solutions over $F$. 

(ii) Take any $v,v'\in V_F\subseteq  V_L$ with $G\cdot v\neq G\cdot v'$. Then there exists an 
$f\in L[V_L]^G_d$ with $d\le \sepbeta(G,V_L)$ such that $f(v)\neq f(v')$. 
It follows by (i) that there exists an $h\in F[V]^G_d$ with $h(v)\neq h(v')$. 
This clearly shows the desired inequality. 

(iii) is an immediate consequence of (ii). 
\end{proof}

\begin{proposition}\label{prop:abelian-field-dependence} 
Let $G$ be a finite abelian group. 
Then we have 
\begin{align}\label{eq:sepbetaK<sepbetaC}
\sepbeta^\field(G)\le \sepbeta^{\mathbb{C}}(G), 
\text{ with equality when }
 \field 
\text{ contains an element } 
\\ \notag \text{whose multiplicative order 
equals the exponent of }G. 
\end{align} 
 \end{proposition} 

\begin{proof} 
In \cite[Corollary 2.6]{domokos:abelian}, a characterization purely in terms of 
$G$ (and not referring to $\field$) is given for $\sepbeta^\field(G)$, valid when 
$\field$ contains an element 
whose multiplicative order equals the exponent of $G$ (in fact \cite{domokos:abelian} 
assumes that the base field is algebraically closed, but the proof of the quoted result 
works verbatim under the present weaker assumption on $\field$). 
So $\sepbeta^\field(G)=\sepbeta^{\mathbb{C}}(G)$ when 
$\field$ contains an element of multiplicative order $|G|$. 
In particular, for any $\field$ we have $\sepbeta^{\overline{\field}}(G)=\sepbeta^\mathbb{C}(G)$, 
where $\overline{\field}$ is the algebraic closure  
of $\field$ (recall the running assumption that $|G|$ 
is invertible in $\field$). 
Note finally that by 
Lemma~\ref{lemma:spanning invariants} (iii) we have 
$\sepbeta^\field(G)\le \sepbeta^{\overline{\field}}(G)$.  
\end{proof}

\subsection{Failure of strict monotonicity}

For a normal subgroup $N$ of $G$, 
all representations of $G/N$ can be viewed as a representation of $G$ whose kernel contains $N$. 
Therefore we have the obvious inequalities  
\begin{equation}\label{eq:sepbeta(G/N)}
\beta^\field(G)\ge \beta^\field(G/N) 
\text{ and }\sepbeta^\field(G)\ge \sepbeta^\field(G/N).
\end{equation}
Moreover,  both the Noether number and the separating Noether number are monotone for subgroups as well; 
that is, 
for any subgroup $H$ of a finite group $G$ we have 
\begin{equation}\label{eq:beta(H)}
\beta^\field(G)\ge \beta^\field(H) 
\mbox{ (see \cite{schmid}) and }
\sepbeta^\field(G)\ge \sepbeta^\field(H) 
\mbox{ (see  \cite[Theorem B]{kohls-kraft}).} 
\end{equation} 
The first inequality \eqref{eq:beta(H)} was sharpened in 
\cite[Theorem 1.2]{cziszter-domokos:lower bound}, 
where the following strict monotonicity of the Noether number was proved: 

\begin{align}\label{eq:strict monotonicity} 
    \beta^\field(G)>\beta^\field(H) \text{ for any subgroup }H\neq G; 
    \\ \notag 
    \beta^\field(G)>\beta^\field(G/N) \text{ for any normal subgroup }N\neq \{1_G\} \text{ of }G. 
    \end{align}

The results of Theorem~\ref{thm:sepbeta index two}
and 
Theorem~\ref{thm:sepbeta(<32)} 
show that the analogues of \eqref{eq:strict monotonicity} (i.e. strict monotonicity) does not hold for the separating Noether number.  
Indeed, for $\field$ as in the cited theorems, we have 
\[\sepbeta^\field(\mathrm{Dic}_8)=6
=\sepbeta^\field(\mathrm{Dic}_8\times \mathrm{C}_2),\]
and $\mathrm{Dic}_8\times \mathrm{C}_2$ has a proper direct factor isomorphic 
to $\mathrm{Dic}_8$. 
Similar examples exist among the abelian groups. 
For example, we have 
\[\sepbeta^\field(\mathrm{C}_6\oplus \mathrm{C}_6\oplus \mathrm{C}_6)=12=
\sepbeta^\field(\mathrm{C}_6\oplus \mathrm{C}_6\oplus \mathrm{C}_3)\] 
by \cite[Theorem 6.1]{schefler_c_n^r}.

\subsection{Comments on the methods used} 
We make an essential use of the 
known results on exact values of the Noether numbers 
(as summarized in \cite{cziszter-domokos-szollosi}) of  
the groups considered here. In particular, in each  
case when the Noether number and the separating Noether number 
happen to coincide, we just need to come up with a representation and 
two points with distinct orbits that can not be separated by 
invariants of strictly smaller degree than the Noether number. 
In all such cases it turns out that rather small 
dimensional representations suffice to provide such examples. 
On the other hand, to settle the cases when the separating 
Noether number is strictly smaller than the Noether number requires  
elaborate work: 
we need a thorough analysis of 
concrete representations, discussions of stabilizers, 
group automorphisms, 
construction of 
invariants and relative invariants, as well as  
several ad hoc ideas that help to undertand the set of solutions of systems of polynomial equations. What makes these computations 
feasible is that by general principles (see the notion of Helly dimension 
and Lemma~\ref{lemma:helly} below) it is sufficient to deal with representations having a small 
number of irreducible summands. 

\subsection{Organization of the paper} 
In Section~\ref{sec:prel} we introduce notation and 
collect general facts used throughout the paper. 
We turn in Section~\ref{sec:indextwo} to the proof of 
Theorem~\ref{thm:sepbeta index two} and Theorem~\ref{thm:sepbeta(D2nxC2)}. 
In Section~\ref{sec:easy groups} we deal with the remaining 
groups of order less than or equal to $16$ 
for which the separating Noether number equals 
the Noether number, namely with the alternating group $A_4$, the non-abelian 
semi-direct product $(\mathrm{C}_2\times\mathrm{C}_2)\rtimes \mathrm{C_4}$, 
and the Pauli group. In Section~\ref{sec:HxC2} we discuss the smallest non-abelian groups 
whose separating Noether number is strictly smaller than the Noether number, namely 
$\mathrm{Dic}_8\times\mathrm{C}_2$ and $\mathrm{C}_4\rtimes \mathrm{C}_4$. 
In Section~\ref{sec:algorithm} we sketch an algorithm (implemented in SageMath) that can 
compute the separating Noether number of  a small group $G$ over the algebraic closure 
of a chosen base field $\field$ (and hence provides an 
upper bound for the separating Noether number $\sepbeta^\field(G)$). 

\section{Preliminaries} \label{sec:prel}

\subsection{Notational convention.} \label{subsec:convention} 
Denote by $\widehat G$ the group of characters of $G$ (i.e. the group of homomorphisms $G\to \field^\times$, written multiplicatively). 
For $\chi\in \widehat G$ denote by $U_\chi$ the $1$-dimensional 
vector space $\field$ endowed with the  representation $\chi$; 
that is, 
$g\cdot \lambda=\chi(g)\lambda$ for $g\in G$ and $\lambda\in \field$.  
In several cases we shall deal with a $\field G$-module \begin{equation}\label{eq:V+U} 
V=W\oplus U, \qquad W=W_1\oplus\cdots\oplus W_l, \qquad U= U_{\chi_1} \oplus \cdots \oplus U_{\chi_m},
\end{equation}   
where the $W_i$ are pairwise non-isomorphic irreducible $\field G$-modules of dimension at least $2$, the $\chi_1,\dots,\chi_m$ are distinct characters of $G$. We shall denote by $x_1,x_2,\dots$ the coordinate functions on $W_1$, by 
$y_1,y_2,\dots$ the coordinate functions on $W_2$, $z_1,z_2,\dots$ the coordinate functions on $W_3$, etc., 
and denote by $t_\chi$ the coordinate function on $U_\chi$. 
Usually we shall specify a concrete representation of $G$ (in other words, a $\field G$-module) 
by giving a group homomorphism $\psi:G\to \mathrm{GL}_n(\field)$. 
The underlying vector space is then the space of column vectors $V=\field^n$, on which $g\in G$ operates via multiplication 
by the matrix $\psi(g)$. 
The action of $G$ on the variables $x_1,\dots,x_n$ (which form a basis in the dual space $V^*$ of $V$)
is given explicitly by the following formula:  
\[g\cdot x_j=\sum_{i=1}^n\psi(g^{-1})_{ji}x_i,\] 
whereas we have $g\cdot t_\chi=\chi(g^{-1})t_\chi$. 
Associated to the above direct sum decomposition of $V$ is 
an $\mathbb{N}_0^{l+m}$-grading on $\field[V]$: this corresponds to the action 
of the torus $(\field^\times)^{l+m}$ given by 
$(\lambda_1,\dots,\lambda_{l+m})\cdot (w_1,\dots,w_l,u_1,\dots,u_m)=
(\lambda_1 w_1,\dots, \lambda_{l+m}u_m)$. 
The phrase `multihomogeneous' will refer to this multigrading. 
Since the above torus action commutes with the $G$-action on $V$, 
the algebra $\field[V]^G$ is spanned by multihomogenous elements. 
In particular, for $v,v'\in V$ and a positive integer $d$, we have that 
\begin{align}\label{eq:multihomogeneous separating}
v,v'\mbox{ can be separated by a homogeneous invariant of degree }d
\\ \notag \iff v,v'\mbox{ can be separated by a multihomogeneous invariant of degree }d. 
\end{align}

We shall write $A^T$ for the transpose of the matrix (e.g. row vector) $A$. 

\subsection{Relative invariants} \label{sec:Davenport} 

An element $f\in \field[V]$ is a \emph{relative invariant of weight $\chi\in \widehat G$} if 
$g\cdot f=\chi(g^{-1})f$ for all $g\in G$ 
(so $f(gv)=\chi(g)f(v)$ for all $g\in G$ and $v\in V$). We set 
\[\field[V]^{G,\chi}:=\{f\in \field[V]\mid f \text{ is a relative invariant of weight }\chi\}.\] 
For example, if $V$ is as in \eqref{eq:V+U}, then 
the variable $t_\chi$ is a relative invariant of weight $\chi$. 
Any multihomogeneous element in $\field[V]$ is of the form 
$ft$, where $f\in \field[W]$ and $t=\prod_{j=1}^mt_{\chi_j}^{n_j}$. 
Note that $t$ is a relative invariant of weight $\prod_{j=1}^m\chi_j^{n_j}$. 
So $ft$ is a $G$-invariant if and only if $f$ is a relative invariant whose weight is the inverse 
of the weight of $t$. 

\section{Groups with a cyclic subgroup of index two}\label{sec:indextwo} 

Denote by $\mathrm{C}_n$ the cyclic group of order $n$. 
We use the notation for semidirect products 
\[ \mathrm{C}_m \rtimes_d \mathrm{C}_n = \langle a,b \mid  a^m=1, b^n=1, bab^{-1}=a^d \rangle  \quad \text{ where } d \in \mathbb{N}\mbox{ is coprime to }m. \] 
We need the following classification of finite groups with a cyclic subgroup of index $2$ 
(see \cite[Section 10]{cziszter-domokos:indextwo} for details and references):

\begin{proposition}\label{prop:class index 2 cyclic} 
Any finite group containing a cyclic subgroup of index two is isomorphic to
\begin{equation}\label{eq:H}
\mathrm{C}_s\times (\mathrm{C}_r\rtimes_{-1} H)\end{equation} 
where $r,s$ are coprime odd integers, 
$\mathrm{C}_r\rtimes_{-1} H$ is the semidirect product where the kernel of 
the action of $H$ on $\mathrm{C}_r$ contains an index two cyclic subgroup $\langle a  \rangle$ 
of $H$, 
$b\in H\setminus \langle a \rangle$ acts via inversion on $\mathrm{C}_r$ 
(i.e. $bxb^{-1}=x^{-1}$ for all $x\in \mathrm{C}_r$), and 
$H$ is one of the following $2$-groups: 
\begin{itemize}
\item[(i)] $\mathrm{C}_{2^n}$ \quad ($n\geq 1$);
\item[(ii)] $\mathrm{C}_{2^{n-1}}\times \mathrm{C}_2$ \quad  ($n\geq 2$); 
\item[(iii)] $\mathrm{M}_{2^n} := \mathrm{C}_{2^{n-1}} \rtimes_d  \mathrm{C}_2, \qquad d={2^{n-2}+1}$ \quad ($n\geq 3$);
\item[(iv)] $\mathrm{D}_{2^n} 	:= \mathrm{C}_{2^{n-1}} \rtimes_{-1} \mathrm{C}_2$ \quad  ($n\geq 4$);  
\item[(v)] $\mathrm{SD}_{2^n} := \mathrm{C}_{2^{n-1}} \rtimes_d \mathrm{C}_2, \qquad d={2^{n-2}-1}$ \quad ($n\geq 4$);  
\item[(vi)] $\mathrm{Dic}_{2^n} := \langle a,b\mid a^{2^{n-1}}=1, b^2=a^{2^{n-2}}, bab^{-1}= a^{-1}\rangle$ \quad ($n\geq 3$).  
\end{itemize}
\end{proposition}  

Among the groups listed in Proposition~\ref{prop:class index 2 cyclic} are the dicyclic groups. 
For a positive integer $m>1$, the \emph{dicyclic group} is  
\[\mathrm{Dic}_{4m}=\begin{cases} \mathrm{C}_r\rtimes_{-1} \mathrm{Dic}_{2^n} &\text{ for  }m=2^{n-2}r \text{ even }\\
\mathrm{C}_m\rtimes_{-1} \mathrm{C}_4 &\text{ for  }m>1 \text{ odd. }\end{cases}\] 

Note that $\mathrm{Dic}_8$ is the \emph{quaternion group} of order $8$. 

 \begin{proposition}\label{prop:betasep-index2} 
Let $G$ be a non-cyclic group with a cyclic subgroup of index two. 
Assume that 
$\field$ has an element of multiplicative order $|G|$. 
Then
\[ \sepbeta^\field(G) \ge \frac{1}{2} |G|  +
\begin{cases}
2 	& \text{ if } G=\mathrm{Dic}_{4m}, \text{  $m>1$};\\
1	& \text{ otherwise. }
\end{cases}\]
\end{proposition} 

\begin{proof} 
Let $G$ be a group as in \eqref{eq:H} in Proposition~\ref{prop:class index 2 cyclic}, 
where $H$ is a semidirect product of a cyclic group and the two-element group (so $H$ is $\mathrm{C}_2$ or $H$ is of type (ii), (iii), (iv), or (v) from Proposition~\ref{prop:betasep-index2}). 
Then $G$ has a matrix representation generated by the matrices 
\[\   \begin{bmatrix} 
      0 & 1 & 0 \\
      1 & 0 & 0 \\
      0 & 0 & -1
   \end{bmatrix}, \quad 
   \begin{bmatrix} 
      \xi & 0 & 0 \\
      0 & \xi^k & 0 \\
      0 & 0 & 1
   \end{bmatrix}, \quad 
   \begin{bmatrix} 
      \rho & 0 & 0 \\
      0 & \rho^{-1} & 0 \\
      0 & 0 & 1
   \end{bmatrix}, \quad 
   \begin{bmatrix} 
      \varepsilon & 0 & 0 \\
      0 & \varepsilon & 0 \\
      0 & 0 & 1
   \end{bmatrix}, 
       \]
where $k$ is some positive integer coprime to $2^{n-1}$ depending on the type of $H$, and  $\xi$, $\rho$, $\varepsilon$ are roots of $1$ of multiplicative order $2^{n-1}$, 
$r$, $s$. Consider the corresponding algebra $\field[x_1,x_2,t]^G$, where $x_1,x_2,t$ are the coordinate functions on the vector space $\field^3$, on which the group $G$ acts via  the given matrix representation. Setting $m:=sr2^{n-1}$, the monomials $x_1^m$ and $x_2^m$ 
are fixed by the cyclic index two subgroup $A$ of $G$ generated by the second, third, and fourth matrices above. Moreover, they are interchanged by the first matrix. 
It follows that $(x_1^m-x_2^m)t$ is a $G$-invariant. This degree $m+1$ invariant takes different values on the vectors $v:=[1,0,1]^T$ and $v':=[1,0,-1]^T$ in $\field^3$. 

Note that our representation is the direct sum of a $2$-dimensional and a $1$-dimensional representation. 
So by \eqref{eq:multihomogeneous separating} it is sufficient to show that 
 $v$ and $v'$ can not be separated by a \emph{multihomogeneous} invariant  of degree at most $m$. 
Suppose for contradiction that $f(v)\neq f(v')$, where $f$ is a multihomogeneous invariant of degree at most $m$. 
So $f=t^df_1(x_1,x_2)$.  
Since the projections of these two vectors in the $x_1,x_2$ coordinate plane are the same, 
no invariant depending only on $x_1$ and $x_2$ can separate them. Thus we have $d>0$. 
The invariants depending only on $t$ are the polynomials in $t^2$, which take the same value (namely $1$) on these vectors, hence $\deg(f_1)>0$. Since $t$ is $A$-invariant, 
$f_1(x_1,x_2)\in \field[x_1,x_2]^A$, so $f_1$ is 
a $\field$-linear combination of $A$-invariant monomials in $x_1$ and $x_2$. 
The monomials having positive degree in $x_2$ vanish on both $v$ and $v'$. 
It follows that 
$f_1$ must have an $A$-invariant monomial depending only on $x_1$. 
The only such monomials are the powers of $x_1^m$. Thus $\deg(f)\ge m+d>m$, 
a contradiction.  

Assume next that $G$ is a dicyclic group $\mathrm{Dic}_{4m}$. It is  isomorphic to the matrix group 
generated by
  \[\begin{bmatrix} 
      0 & \mathrm{i}\\
      \mathrm{i}& 0 \\
   \end{bmatrix}, \quad 
   \begin{bmatrix} 
      \rho & 0 \\
      0 & \rho^{-1} \\
   \end{bmatrix} \]
   where $\mathrm{i}$ is an element of $\field$ with multiplicative order $4$ (so $\mathrm{i}^2=-1$) 
   and $\rho$ is an element of $\field$ of multiplicative order $2m$. 
   It is an easy exercise to show that the corresponding algebra of invariants is generated by 
   $x_1^{2m}-x_2^{2m}$, $(x_1x_2)^2$, $(x_1^{2m}+x_2^{2m})x_1x_2$ when $m$ is odd, and 
   by $x_1^{2m}+x_2^{2m}$, $(x_1x_2)^2$, $(x_1^{2m}-x_2^{2m})x_1x_2$ when $m$ is even. 
   If $m$ is odd, the vectors $[1,1]^T$ and $[1,-1]^T$ are separated by the degree $2m+2$ generator, 
   but all the smaller degree generators agree on them. 
   By assumption $\field$ has an element $\nu$ of multiplicative order $4m=|G|$.  
For even $m$ the vectors $[\nu,1]^T$ and $[\nu,-1]^T$ are separated by the degree 
   $2m+2$ generator, whereas the smaller degree generators coincide on them.

   A matrix representation of $\mathrm{C}_s\times \mathrm{Dic}_{4m}$ (where $s>1$ is odd and is co-prime to $m$) 
   is generated by
   \[   \begin{bmatrix} 
      0 & \mathrm{i}& 0 \\
      \mathrm{i}& 0 & 0 \\
      0 & 0 & -1
   \end{bmatrix}, \quad 
   \begin{bmatrix} 
      \rho & 0 & 0 \\
      0 & \rho^{-1} & 0 \\
      0 & 0 & 1
   \end{bmatrix}, \quad 
   \begin{bmatrix} 
      \varepsilon & 0 & 0 \\
      0 & \varepsilon & 0 \\
      0 & 0 & 1
   \end{bmatrix},\] 
where $\rho$ and $\varepsilon$ are roots of $1$ with multiplicative order $2m$ and $s$.    
The vectors $v:=[1,0,1]^T$ and $v':=[1,0,-1]^T$ are separated by the invariant 
$(x_1^{2ms}-x_2^{2ms})t$ when $m$ is odd and by $(x_1^{2ms}+x_2^{2ms})t$ when $m$ is even. 
On the other hand, any invariant of degree at most $2ms$ takes the same value on $v$ and $v'$ (this can be seen similarly to the argument in the second paragraph of the proof). 

It remains to deal with the case when 
$G\cong \mathrm{C}_s\times (\mathrm{C}_r\rtimes_{-1} \mathrm{C}_{2^n})$ where $n\ge 3$. 
A matrix representation of this group is generated by 
\[   \begin{bmatrix} 
      0 & \omega & 0 \\
      \omega & 0 & 0 \\
      0 & 0 & -1
   \end{bmatrix}, \quad 
   \begin{bmatrix} 
      \rho & 0 & 0 \\
      0 & \rho^{-1} & 0 \\
      0 & 0 & 1
   \end{bmatrix}, \quad 
   \begin{bmatrix} 
      \varepsilon & 0 & 0 \\
      0 & \varepsilon & 0 \\
      0 & 0 & 1
   \end{bmatrix},\] 
where $\omega$, $\rho$ and $\varepsilon$ are roots of $1$ with multiplicative order $2^n$, $r$, and $s$.   The square of the first matrix together with the second and the third generate a cyclic subgroup $A$ of $G$ of order $m:=2^{n-1}rs$, acting on the line spanned by 
$[1,0,0]^T$ via a character of order $m$. The invariant $(x_1^m+x_2^m)t$ separates 
$v:=[1,0,1]^T$ and $v':=[1,0,-1]^T$, and any invariant of degree at most $m$ takes the same value on $v$ and $v'$ (this can be seen similarly to the argument in the second paragraph of the proof). 
\end{proof}

\begin{proofof}{Theorem~\ref{thm:sepbeta index two}} 
The Noether number $\beta^\field(G)$ was computed for all groups with a cyclic subgroup of index $2$ in \cite[Theorem 10.3]{cziszter-domokos:indextwo}, 
where (as a special case of a statement on the so-called generalized Noether number) 
it was shown that 
if $G$ is a non-cyclic group with a cyclic subgroup of index two 
and $\mathrm{char}(\field)$ is not a divisor of $|G|$, then
\begin{equation}\label{eq:beta(index two cyclic)} 
\beta^\field(G) = \frac{1}{2} |G|  +
\begin{cases}
2 	& \text{ if } G=\mathrm{Dic}_{4m}, \text{  $m>1$};\\
1	& \text{ otherwise. }\end{cases}
\end{equation} 
Therefore the desired result on the separating Noether number  is an immediate consequence of Proposition~\ref{prop:betasep-index2}, the obvious inequality $\sepbeta^\field(G)\le \beta^\field(G)$, 
and \eqref{eq:beta(index two cyclic)}.  
\end{proofof}

\subsection{The groups $\mathrm{D}_{2n}\times \mathrm{C}_2$}  
In this section $n\ge 2$ is even, and 
\[G=\mathrm{D}_{2n}\times \mathrm{C}_2=\langle a,b,c\mid a^n=b^2=c^2=1,\ bab=a^{-1},\ ac=ca,\ bc=cb\rangle.\] 

\begin{proposition}\label{prop:D2nxC2} 
Assume that $\field$ 
contains an element of multiplicative order $n$. 
Then we have the inequality $\sepbeta^\field(G)\ge n+2$. 
\end{proposition} 

\begin{proof} 
Let $\rho$ denote an element of $\field$ 
   of multiplicative order $n$. 
Take the direct sum of the representations 
\[ a\mapsto \begin{bmatrix} 
      \rho & 0 \\
      0 & \rho^{-1} \\
   \end{bmatrix}, \quad 
   b\mapsto \begin{bmatrix} 
      0 & 1 \\
      1 & 0 \\
   \end{bmatrix}, \quad 
   c\mapsto \begin{bmatrix} 
      1 & 0 \\
      0 & 1 \\
   \end{bmatrix},  \]
\[a\mapsto 1, \quad b\mapsto -1, \quad c\mapsto -1
   \qquad \text{ and } \qquad a\mapsto 1, \quad b\mapsto 1, \quad c\mapsto -1.\] 
Recall that $x_1,x_2$ stand for the standard coordinate functions  on the $2$-dimensional summand, and $t_1,t_2$ stand for the coordinate functions on the $1$-dimensional summands. 
The invariant $(x_1^n-x_2^n)t_1t_2$ separates the points 
$v:=([1,0]^T,1,1)$ and $v':=([1,0]^T,1,-1)$ in $\field^4$. 
It is sufficient to show that $v$ and $v'$ can not be separated by an invariant of degree at most $n+1$. 
Suppose for contradiction that $h(v)\neq h(v')$ for some homogeneous 
$h\in \field[x_1,x_2,t_1,t_2]^G$ of degree at most $n+1$. We may assume by \eqref{eq:multihomogeneous separating} 
that $h$ is multihomogeneous (so it is homogeneous both in $t_1$ and $t_2$), 
and has minimal possible total degree. So $h=t_1^{k_1}t_2^{k_2}h_1(x_1,x_2)$. 
Since the $(x_1,x_2,t_1)$-coordinates of $v$ and $v'$ are the same, we have $k_2>0$. 
Since $h_1$ is $\langle c\rangle$-invariant, 
$k_1+k_2$ must be even. 
Moreover, $\field[t_1,t_2]^G=\field[t_1^2,t_2^2]$, and $t_1^2(v)=1=t_1^2(v')$, 
$t_2^2(v)=1=t_2^2(v')$, so $\deg(h_1)>0$, and (since $h$ has minimal possible degree) $h=t_1t_2h_1$. 
Thus $h_1$ is a relative $G$-invariant of degree at most $n-1$, 
on which $G$ acts via the character 
$a\mapsto 1,$ $b\mapsto -1$, $c\mapsto 1$. The action of $G$ on $\field^2$ factors through the action of the dihedral group $\mathrm{D}_{2n}\cong G/\langle c\rangle$, and it is well known (and easy to see) 
that the smallest degree relative invariant in $\field[x_1,x_2]$ with the above weight is $x_1^n-x_2^n$. 
We reached the desired contradiction. 
\end{proof} 

\begin{proofof}{Theorem~\ref{thm:sepbeta(D2nxC2)}}
We have the isomorphism 
\[\mathrm{D}_{2n}\times \mathrm{C}_2\cong (\mathrm{C}_n\times \mathrm{C}_2)\rtimes_{-1}\mathrm{C}_2.\]
Therefore by \cite[Corollary 5.5]{cziszter-domokos:indextwo} we have 
\[\beta^\field(\mathrm{D}_{2n}\times \mathrm{C}_2)= \mathsf{D}(\mathrm{C}_n\times \mathrm{C}_2)+1,\] 
where $\mathsf{D}(\mathrm{C}_n\times \mathrm{C}_2)$ is the Davenport constant of the 
abelian group $\mathrm{C}_n\times \mathrm{C}_2$ (the maximal length of an irreducible product-one 
sequence over $\mathrm{C}_n\times \mathrm{C}_2$). 
By \cite{olson2} we have $\mathsf{D}(\mathrm{C}_n\times \mathrm{C}_2)=n+1$, 
implying in turn that $\sepbeta^\field(\mathrm{D}_{2n}\times \mathrm{C}_2)\le n+2$. 
The reverse inequality holds by\ Proposition~\ref{prop:D2nxC2}. 
\end{proofof}

\section{Further groups with order $\leq 16$}\label{sec:easy groups} 

\subsection{The alternating group $\mathrm{A}_4$}

\begin{proposition} \label{prop:alt_n} 
For $n\ge 4$ we have the inequality $\sepbeta^\field(\mathrm{A}_n)\ge n(n-1)/2$. 
\end{proposition} 

\begin{proof} 
Our running assumption that $\mathrm{char}(\field)$ does not divide $|G|$ forces for $G=\mathrm{A}_n$ 
that $|\field|\ge n$ if $n\ge 4$. 
Therefore $\field$ has $n$ distinct elements $v_1,\dots,v_n$.   
The vectors $v:=[v_1,v_2,v_3,\dots,v_n]^T$ and $v':=[v_2,v_1,v_3,v_4,\dots,v_n]^T$ belong to the same orbit under the symmetric group $\mathrm{S}_n$ acting on $\field^n$ via permuting the coordinates, but $v$ and $v'$ have different orbits with respect to the alternating subgroup 
$\mathrm{A}_n$.  It is well known that the corresponding algebra of invariants is 
\[\field[x_1,\dots,x_n]^{\mathrm{A}_n}=\field[x_1,\dots,x_n]^{\mathrm{S}_n}
\oplus \Delta \field[x_1,\dots,x_n]^{\mathrm{S}_n},\] 
where $\Delta:=\sum_{1\le i<j\le n}(x_i-x_j)$. This shows that all $\mathrm{A}_n$-invariants of degree less than $n(n-1)/2$ are in fact $\mathrm{S}_n$-invariants. So they can not separate $v$ and $v'$, which have different $\mathrm{A}_n$-orbits. 
\end{proof} 

It is known that $\beta^\field(\mathrm{A}_4)=6$ 
(see \cite[Theorem 3.4]{CzD:1}), hence Proposition~\ref{prop:alt_n} 
has the following consequence: 

\begin{theorem}\label{thm:sepbeta(A4)} 
We have the equality 
$\sepbeta^\field(\mathrm{A}_4)=6$.
\end{theorem} 

\subsection{The non-abelian group $(\mathrm{C}_2\times \mathrm{C}_2)\rtimes \mathrm{C}_4$} 

In this section 
\[G:=\langle a,b,c\mid a^2=b^2=c^4=1, \quad ab=ba, \quad cac^{-1}=b, \quad cbc^{-1}=a\rangle.\] 

\begin{theorem} \label{thm:sepbeta((C2xC2)rtimesC4)}
Assume that $\field$ has an element $\mathrm{i}$ of multiplicative order $4$. 
Then we have the equality 
$\sepbeta^\field((\mathrm{C}_2\times\mathrm{C}_2)\rtimes \mathrm{C}_4)=6$. 
\end{theorem} 

\begin{proof} 
 The equality $\beta^\field(G)=6$ is proved in \cite[Proposition 3.10]{cziszter-domokos-szollosi}, so it is sufficient 
 to prove that $\sepbeta^\field(G)\ge 6$. 
The element $c^2$ belongs to the center of $G$, and the factor group $G/\langle c^2\rangle$ is 
isomorphic to the dihedral group $\mathrm{D}_8$. 
The representation 
\[ a\mapsto \begin{bmatrix} 
      0 & 1 \\
      1 & 0 \\
   \end{bmatrix}, \quad 
   b\mapsto \begin{bmatrix} 
      0 & -1 \\
      -1 & 0 \\
   \end{bmatrix}, \quad 
   c\mapsto \begin{bmatrix} 
      1 & 0 \\
      0 & -1 \\
   \end{bmatrix} \]
   has $c^2$ in its kernel and gives the defining representation of $\mathrm{D}_8$ as the group of symmetries of a square in the euclidean plane.  
   The element $ab$ also belongs to the center of $G$, and the factor group $G/\langle ab\rangle \cong \mathrm{C}_4\times \mathrm{C}_2$ is generated by the coset of  $c$ and the coset of $a$. 
   Thus $G$ has the $1$-dimensional representations 
   \[a\mapsto 1, \quad b\mapsto 1, \quad c\mapsto \mathrm{i}
   \qquad \text{ and } \qquad a\mapsto -1, \quad b\mapsto -1, \quad c\mapsto \mathrm{i}.\] 
   Take the direct sum of the  above $2$-dimensional and $1$-dimensional representations, 
 and denote by $x_1,x_2$ the standard coordinate functions  on the $2$-dimensional summand, and by $t_1,t_2$ the coordinate functions on the $1$-dimensional summands. 
   The corresponding ring $\field[x_1,x_2,t_1,t_2]^G$ of invariants contains 
   the invariant $h:=(x_1^2-x_2^2)t_1^3t_2$, separating the points $v:=([1,0]^T,1,1)$ and $v':=([1,0]^T,1,-1)$  in $\field^4$. 
   By \eqref{eq:multihomogeneous separating}, in order to prove $\sepbeta^\field(G)\ge 6$, it is sufficient to show that 
 $v$ and $v'$ can not be separated by a multihomogeneous invariant of degree at most $5$. 
Suppose for contradiction that $f(v)\neq f(v')$, where $f$ is a multihomogeneous invariant of degree at most $5$. It has positive degree in $t_2$, because 
  the $(x_1,x_2,t_1)$ coordinates of $v$ and $v'$ are the same. Any $G$-invariant depending only on 
  $t_1$ and $t_2$ is a  polynomial in $t_1^4$, $t_1^2t_2^2$, $t_2^4$, and all these polynomials take the value $1$ both at $v$ and $v'$. 
  So $f=t_1^{k_1}t_2^{k_2}f_1(x_1,x_2)$, 
  where $0<k_1+k_2\le 4$, $\deg(f_1)>0$, and $k_2$ is odd. 
  Note that $t_1^{k_1}t_2^{k_2}$ is a relative $G$-invariant, hence $f_1(x_1,x_2)$ must be a relative $G$-invariant, on which $G$ acts by the inverse character. 
  Since $c^2$ fixes $f_1$, it must fix also $t_1^{k_1}t_2^{k_2}$, so $k_1+k_2$ is even. 
  Consequently, $(k_1,k_2)\in \{(1,1),(3,1),(1,3)\}$, and $f_1\in \field[x_1,x_2]$ is a relative 
 $\mathrm{D}_8\cong G/\langle c^2\rangle$-invariant of degree at most $5-(k_1+k_2)$. 
 There is no degree $1$ relative $G$-invariant in $\field[x_1,x_2]$, hence $t_1^{k_1}t_2^{k_2}=t_1t_2$, 
 and so $f$ is a relative invariant of degree at most $3$. 
 It is easy to see that up to non-zero scalar multiples, the only relative $G$-invariants in 
 $\field[x_1,x_2]_{\le 3}$ with non-trivial weight are $x_1x_2$ and $x_1^2-x_2^2$, but none of 
 $x_1x_2t_1t_2$ or $(x_1^2-x_2^2)t_1t_2$ is a $G$-invariant. 
 We arrived at the desired contradiction, finishing the proof. 
 \end{proof} 

\subsection{The Pauli group}
Assume that the field $\field$ has an element $\mathrm{i}$ of multiplicative order $4$.
In this section $G$ stands for the subgroup of $\mathrm{GL}(\field^2)$ 
generated by the matrices 
\[\begin{bmatrix} 
      0 & 1 \\
      1 & 0 \\
   \end{bmatrix}, \quad 
   \begin{bmatrix} 
      0 & -\mathrm{i} \\
      \mathrm{i} & 0 \\
   \end{bmatrix}, \quad 
   \begin{bmatrix} 
      1 & 0 \\
      0 & -1 \\
   \end{bmatrix}.  \] 

\begin{theorem} \label{thm:sepbeta(Pauli)}
Assume that $\field$ has an element of multiplicative order $4$. 
We have the equality 
 $\sepbeta^\field(G)=7$. 
 \end{theorem} 
 \begin{proof}
Since $G$ is generated by reflections on the space $\field^2$, by the Sheppard-Todd-Chevalley Theorem (see e.g. 
\cite[Theorem 7.2.1]{benson} or 
\cite[Section 3.9.4]{derksen-kemper}), 
$\field[x_1,x_2]^G$ is a polynomial algebra with two generators.  In fact 
it is generated by $x_1^4+x_2^4$ and $(x_1x_2)^2$. 
Moreover, the factor of $\field[x_1,x_2]$ modulo the Hilbert ideal (the ideal generated by 
 $x_1^4+x_2^4$ and $(x_1x_2)^2$) is isomorphic to the regular representation of $G$ (see \cite{chevalley}). 
 The element $f:=(x_1^4-x_2^4)x_1x_2$ has the property that $g\cdot f=-f$ where $g$ is any of the above three generators of $G$. One can easily see that $f$ does not belong to the Hilbert ideal. Take an extra indeterminate $t$ and extend the action of  
 $G$ to $\field[x_1,x_2,t]$ by setting $g\cdot t=-t$ for any of the above three generators of $G$. 
 Then $ft$ is a degree $7$ $G$-invariant separating $v:=[1,1,1]^T$ and $v':=[1,1,-1]^T$. 
 We claim that $v$ and $v'$ can not be separated by an invariant of degree at most $6$. 
 Suppose for contradiction that $h$ is a multihomogeneous 
 $G$-invariant of degree at most $6$ 
 with $h(v)\neq h(v')$ (cf. \eqref{eq:multihomogeneous separating}). Then $h$ is homogeneous in $t$. 
 Since $v$ and $v'$ have the same $(x_1,x_2)$ coordinates, any element of $\field[x_1,x_2]$ takes the same value on them, so $h$ has positive degree in $t$. Moreover, 
 since $t^2$ 
 is $G$-invariant, we may assume that $h=th_1$ for some $h_1\in \field[x_1,x_2]$, where $h_1$ 
 is homogeneous of degree at most $5$. By $G$-invariance of $h$ we must have 
 that $g\cdot h_1=-h_1$ for any of the above three generators $g$ of $G$. 
 However, this $1$-dimensional representation of $G$ occurs only once as a summand in 
 the factor of $\field[x_1,x_2]$ modulo the Hilbert ideal, and we found already one occurrence in the 
 degree $6$ homogeneous component. This contradiction proves the inequality $\sepbeta^\field(G)\ge 7$. 
 The reverse inequality $\sepbeta^\field(G)\le 7$ follows from the equality $\beta^\field(G)=7$ proved in \cite[Example 5.4]{cziszter-domokos-geroldinger}. 
\end{proof}

\section{The smallest non-abelian groups  with $\sepbeta^\field(G)<\beta^\field(G)$}
\label{sec:HxC2}

\subsection{Reduction to multiplicity-free representations and the Helly dimension}\label{sec:helly} 

We record a consequence of the results of \cite{draisma-kemper-wehlau} and 
Lemma~\ref{lemma:spanning invariants}: 

\begin{lemma} \label{lemma:multfree} 
Let $V_1,\dots,V_q$ be a complete irredundant list of representatives of the isomorphism classes of simple $\field G$-modules. 
Assume that 
\[|K|>(\max\{\dim(V_j)\mid j=1,\dots,q\}-1)|G|.\] 
Then we have the equality 
\[\sepbeta^\field(G)=\sepbeta(G,V_1\oplus\cdots
\oplus V_q).\] 
\end{lemma} 

\begin{proof} An arbitrary $\field G$-module $V$ is isomorphic to 
$V_1^{m_1}\oplus\cdots\oplus V_q^{m_q}$, where 
$W^m$ stands for the direct sum $W\oplus\cdots \oplus W$ of $m$ copies of $W$. 
By \cite[Corollary 2.6]{draisma-kemper-wehlau} we have 
$\sepbeta(G,V)\le \sepbeta(G,V_1^{\dim(V_1)}\oplus\cdots\oplus V_q^{\dim(V_q)})$, 
and by \cite[Theorem 3.4 (ii)]{draisma-kemper-wehlau} and 
\cite[Proposition 3.3 (ii)]{draisma-kemper-wehlau} we have 
$\sepbeta(G,V_1^{\dim(V_1)}\oplus\cdots\oplus V_q^{\dim(V_q)})=
\sepbeta(G,V_1\oplus \cdots \oplus V_q)$. 
\end{proof} 

Therefore to compute $\sepbeta^\field(G)$ it is sufficient to deal with multiplicity-free representations of $G$ (unless possibly if $K$ is a finite field with too few elements).  
Lemma~\ref{lemma:multfree} can be improved for groups whose Helly dimension is 
'small'. The \emph{Helly dimension} $\kappa(G)$ is the minimal positive integer 
$k$ having the following property: any set of cosets in $G$ with empty intersection contains a subset of at most $k$ cosets with empty intersection; this quantity was introduced in \cite[Section 4]{domokos:typical}. 
Here we introduce a refinement, yielding sharper results in the present applications. 
Let $V=V_1\oplus\cdots\oplus V_n$ be a $\field G$-module with a fixed direct sum 
decomposition. Denote by $\kappa(G,V)$ the minimal 
positive integer $k$ such that if $C_1,\dots,C_n$ are cosets 
in $G$, where $C_i$ is a left coset with respect to the stabilizer subgroup in 
$G$ of some element $v_i$ in $V_i$, and $\cap_{i=1}^nC_i=\emptyset$, 
then there exists $1\le i_1<\cdots<i_k\le n$ with $\cap_{j=1}^kC_{i_j}=\emptyset$. 
Clearly $\kappa(G,V)\le \kappa(G)$. The following statement is a variant 
of \cite[Lemma 4.1]{domokos-szabo}. 

\begin{lemma}\label{lemma:helly} 
Under the assumptions and notation of Lemma~\ref{lemma:multfree}, 
there exists a $k\le\kappa(G,V_1\oplus\cdots \oplus V_q)$ and 
$1\le j_1<\dots<j_k\le q$ with 
$\sepbeta^\field(G)=\sepbeta(G,V_{j_1}\oplus\cdots\oplus V_{j_k})$. 
\end{lemma} 
\begin{proof} 
By Lemma~\ref{lemma:multfree} we have 
$\sepbeta^\field(G)=\sepbeta(G,V)$, where $V=V_1\oplus\cdots\oplus V_q$. 

Assume that $d\ge \sepbeta(G,V_{i_1}\oplus\cdots \oplus V_{i_k})$ 
for all $1\le i_1<\dots<i_k\le q$, where $k=\kappa(G,V)$. 
Take $v=(v_1,\dots,v_q)$ and $v'=(v'_1,\dots,v'_q)$ such that 
$f(v)=f(v')$ for all $f\in \field[V]^G_{\le d}$. 
Then $(v_{i_1},\dots,v_{i_k})$ and $(v'_{i_1},\dots,v'_{i_k})$ 
belong to the same $G$-orbit in $V_{i_1}\oplus\cdots\oplus V_{i_k}$ 
by the choice of $d$. In particular, $G\cdot v_j=G\cdot v'_j$ 
for $j=1,\dots,q$. Set $C_j:=\{g\in G\mid g\cdot v_j=v'_j\}$, 
then $C_j$ is a left coset in $G$ with respect to the stabilizer subgroup in 
$G$ of $v_j$. Moreover, $\cap_{j=1}^kC_{i_j}\neq \emptyset$ 
(an element $g$ with $g\cdot (v_{i_1},\dots,v_{i_k})=
(v'_{i_1},\dots,v'_{i_k})$ belongs to this intersection). 
This holds for all $1\le i_1<\cdots<i_k\le q$, 
hence by definition of $\kappa(G,V)$ we conclude that 
$\cap_{j=1}^qC_j\neq \emptyset$. For $g\in \cap_{j=1}^qC_j\neq \emptyset$ 
we have $g\cdot v=v'$. This shows the inequality 
$\sepbeta(G,V)\le d$. In particular, 
\begin{equation}\label{eq:j1-jk}
\sepbeta(G,V)\le \sepbeta(G,V_{j_1}\oplus\cdots \oplus V_{j_k}),
\end{equation}
where 
$\sepbeta(G,V_{j_1}\oplus\cdots \oplus V_{j_k})=\max\{\sepbeta(G,V_{i_1}\oplus\cdots \oplus V_{i_k})\mid 
1\le i_1<\dots<i_k\le q\}$. 

The  projection of $V$ onto its direct summand $V_{j_1}\oplus\cdots \oplus V_{j_k}$ 
induces a degree-preserving embedding of 
$\field[V_{j_1}\oplus\cdots \oplus V_{j_k}]^G$ into $\field[V]^G$, 
hence $\sepbeta(G,V_{j_1}\oplus\cdots \oplus V_{j_k})\le \sepbeta(G,V)$, 
showing the reverse inequality to \eqref{eq:j1-jk}. 
\end{proof}

In view of Lemma~\ref{lemma:helly}, it is helpful to find upper bounds on the Helly dimension. By \cite[Lemma 4.2]{domokos:typical} we have the inequality 
\begin{equation}\label{eq:kappa-mu}
\kappa(G)\le \mu(G)+1,
\end{equation} 
where $\mu(G)$ stands for the maximal size of an intersection independent set of subgroups of $G$. Recall that a set of subgroups of $G$ is \emph{intersection independent} if none contains the intersection of the others. 

\begin{lemma}\label{lemma:prime power} 
If a set $S$ of intersection independent subgroups in a group contains a cyclic subgroup of prime power order, then $|S|\le 2$.  
\end{lemma}
\begin{proof} 
Let $S$ be an intersection independent set of subgroups of a group $G$ with $|S|\ge 2$ and $H\in S$ cyclic with $|H|$ a prime power. Take $J\in S$ with $J\cap H$ of minimal possible order. 
Since the lattice of subgroups of $H$ is a chain, for any $L\in S$ we have $L\cap H\supseteq H\cap J$, hence in particular $L\supseteq H\cap J$. Since $S$ is intersection independent, it follows that $L\in \{H,J\}$. As $L\in S$ was arbitrary, it means that 
$S=\{H,J\}$.  
\end{proof} 

\subsection{The use of automorphisms}
The natural action of the automorphism group of $G$ on the isomorphism classes 
of representations of $G$ helps to save some computations later, 
thanks to the following statements: 

\begin{lemma}\label{lemma:auto}
Let $(V,\rho)$ be a $\field G$-module and $\alpha$ an automorphism of the group $G$. 
\begin{itemize} 
\item[(i)] For any $\chi\in \widehat G$ we have 
\[\field[(V,\rho)]^{G,\chi}= 
\field[(V,\rho\circ\alpha)]^{G,\chi\circ\alpha}.\]  
In particular, 
\[\field[(V,\rho)]^G= 
\field[(V,\rho\circ\alpha)]^G.\] 
\item[(ii)] We have 
\[\beta(G,(V,\rho))=\beta(G,(V,\rho\circ\alpha))
\text{ and }\sepbeta(G,(V,\rho))=\sepbeta(G,(V,\rho\circ\alpha)).\]
\item[(iii)] Assume that $(V,\rho)\cong (V,\rho\circ\alpha)$ as $\field G$-modules, so there exists a 
$\field$-vector space isomorphism $T:V\to V$ with $T\circ \rho(g)=(\rho\circ\alpha)(g)\circ T$ 
for all $g\in G$. 
Then the map $f\mapsto f\circ T$ induces a $\field$-vector space isomorphism 
\[\field[(V,\rho)]^{G,\chi}\cong \field[(V,\rho)]^{G,\chi\circ\alpha}.\] 
\end{itemize}
\end{lemma}

\begin{proof} (i) and (ii): The subset $\rho(G)$ of $\mathrm{GL}(V)$ coincides with the subset 
$(\rho\circ\alpha)(G)$ of $\mathrm{GL}(V)$, therefore the representations 
$\rho$ and $\rho\circ\alpha$ yield the same partition of $V$ into the union of $G$-orbits, 
and the same polynomials in $\field[V]$ are invariants (respectively relative invariants) 
for the $G$-action given by $\rho$ as for the $G$-action given by $\rho\circ\alpha$.  
Moreover, if $f\in \field[(V,\rho)]^{G,\chi}$, $g\in G$ and $v\in V$, 
then $f((\rho\circ\alpha)(g)(v))=f(\rho(\alpha(g))(v)=\chi(\alpha(g))f(v)=(\chi\circ\alpha)(g)f(v)$. 
So if a relative invariant $f\in \field[V]$ has weight $\chi$ with respect to $\rho$, 
then it has weight $\chi\circ\alpha$ with respect to $\rho\circ\alpha$. 
This explains both statements. 

(iii): For each $g\in G$ we have 
\[(f\circ T)(gv)=f((T\circ\rho(g))(v))
=f((\rho\circ\alpha)(g)(T(v))
=\chi(\alpha(g))f(T(v))=
(\chi\circ\alpha)(g)(f\circ T)(v). \]
\end{proof}

\subsection{The group $\mathrm{Dic}_8\times \mathrm{C}_2$} 

In this section 
\[G=\mathrm{Dic}_8\times \mathrm{C}_2=\langle a,b\mid a^4=1,\ b^2=a^2,\ ba=a^3b\rangle 
\times \langle c\mid c^2=1\rangle\] 
is the direct product of the quaternion group of order $8$ and the group of order $2$.  
Assume that $\field$ has an element $\mathrm{i}$ of multiplicative order $4$. 
Then $G$ has two irreducible $2$-dimensional representations (up to isomorphism), 
namely 
\[ \psi_1: a\mapsto 
\begin{bmatrix} 
      \mathrm{i}& 0 \\
      0 & -\mathrm{i}\\
   \end{bmatrix}, \quad 
   b\mapsto \begin{bmatrix} 
      0 & -1 \\
      1 & 0 \\
   \end{bmatrix}, \quad 
   c\mapsto \begin{bmatrix} 
      1 & 0 \\
      0 & 1 \\
   \end{bmatrix} \]
(this is the only $2$ dimensional representation of $\mathrm{Dic}_8$ lifted to $G$ by the projection 
from $G$ to its direct factor $\mathrm{C}_2$), 
and 
\[\psi_2: a\mapsto 
\begin{bmatrix} 
      \mathrm{i}& 0 \\
      0 &-\mathrm{i}\\
   \end{bmatrix}, \quad 
   b\mapsto \begin{bmatrix} 
      0 & -1 \\
      1 & 0 \\
   \end{bmatrix}, \quad 
   c\mapsto \begin{bmatrix} 
      -1 & 0 \\
      0 & -1 \\
   \end{bmatrix},\]
   the representation $\psi_1$ tensored with the non-trivial representation  $c\mapsto -1$ of $\mathrm{C}_2$.  
   The other irreducible representations of $G$ are $1$-dimensional, and can be labeled by 
  the elements of the subgroup $\widehat G:=\{\pm1\}\times \{\pm1\}\times\{ \pm1\}$ of $\field^\times\times \field^\times\times \field^\times$ as follows: 
  identify $\chi=(\chi_1,\chi_2,\chi_3)\in \widehat G$ with the representation $\chi:G\to \field^\times$ given by 
  \[\chi:a\mapsto \chi_1,\quad b\mapsto \chi_2, \quad c\mapsto \chi_3.\] 
 Denote by $W_j$ the vector space $\field^2$ endowed with the  representation $\psi_j$ for $j=1,2$, and for $\chi\in \widehat G$ denote by $U_\chi$ the vector space $\field$ endowed with the  representation $\chi$. Set 
 \[U:=\bigoplus_{\chi\in \widehat G}U_\chi.\]  
 The coordinate functions on $W_1$ are denoted by $x_1,x_2$, the coordinate functions 
 on $W_2$ are $y_1,y_2$, and the coordinate function on $U_\chi$ is $t_\chi$ for $\chi\in \widehat G$ 
 (see Section~\ref{subsec:convention}). 
 
\begin{proposition}\label{prop:Vi+U}
For $i=1,2$ we have 
$\beta(G,W_i \oplus U)=6$. 
\end{proposition} 
\begin{proof} First we deal with the case $i=1$. 
It is well known (see for example the proof of Proposition~\ref{prop:betasep-index2}) that 
 \begin{equation}\label{eq:Dic8 mingen}
 \field[W_1]^G=\field[(x_1x_2)^2,\ x_1^4+x_2^4,\ x_1x_2(x_1^4-x_2^4)].\end{equation} 
 Denote by $\mathcal{H}(G,W_1)$ the ideal in $\field[W_1]$ generated by the 
 above three  generators of $\field[W_1]^G$ (called the 
\emph{Hilbert ideal}). 
 It is easy to verify that a Gr\"obner basis with respect to the lexicographic 
 term order induced by 
 $x_1>x_2$ of the Hilbert ideal $\mathcal{H}(G,W_1)$ is 
 $x_1^2x_2^2$, $x_1^4+x_2^4$, $x_2^6$, $x_1x_2^5$. 
Using this Gr\"obner basis one can show that setting
 \[s_{(1,-1,1)}:=x_1x_2,\quad s_{(-1,1,1)}:=x_1^2+x_2^2, \quad s_{(-1,-1,1)}:=x_1^2-x_2^2,\]
  a $\field G$-module direct complement in $\field[x_1,x_2]$ of the Hilbert ideal 
 $\mathcal{H}(G,W_1)$ is 
 \goodbreak
\begin{align*}\field\oplus\mathrm{Span}_\field\{x_1,x_2\}\oplus \field s_{(1,-1,1)}\oplus \field s_{(-1,1,1)}\oplus 
\field s_{(-1,-1,1)}\oplus \mathrm{Span}_\field\{x_1^3,x_2^3\}
\\ \oplus \mathrm{Span}_\field\{x_1^2x_2,x_1x_2^2\} 
\oplus \field s_{(1,-1,1)}s_{(-1,1,1)}\oplus \field s_{(1,-1,1)}s_{(-1,-1,1)}
\\ \oplus \field s_{(-1,1,1)}s_{(-1,-1,1)}
\oplus \mathrm{Span}_\field\{x_1^5,x_2^5\}.  
\end{align*}
The direct summands above are minimal $G$-invariant subspaces,  
and $s_\chi\in \field[W_1]^{G,\chi}$. Recall 
that every multihomogeneous element of $\field[W_1\oplus U]^G$ is of the 
form $ft$, where $f\in \field[W_1]$ and 
$t=\prod_{\chi\in \widehat G}t_\chi^{n_\chi}$ 
is a monomial. 
It follows 
that $\field[W_1\oplus U]^G$ is generated over 
$\field[W_1]^G$ by products of elements from 
\begin{equation}\label{eq:s_chi}
\{s_\chi, \quad t_{\chi'}\mid \chi\in\{(1,-1,1),(-1,1,1),(-1,-1,1)\}, \ \chi'\in \widehat G\},
\end{equation} 
with no $G$-invariant proper subproduct. 
Moreover, it is sufficient to consider products from \eqref{eq:s_chi} 
with at most two factors of the form $s_\chi$, 
because any product of three factors of the form $s_\chi$ 
can be rewritten modulo the Hilbert ideal $\mathcal{H}(G,W_1)$ as a linear combination of shorter products 
of the $s_\chi$. 
Any unordered sequence of $5$ elements from $\widehat G$ has a proper product-one 
subsequence,  
since the Davenport constant of the group $\widehat G$ is $4$ 
(for information about the Davenport constant and its relevance for 
the invariant theory of abelian groups see for example \cite{cziszter-domokos-geroldinger}). 
Thus the algebra $\field[W_1\oplus U]^G$ is generated by 
the generators of $\field[W_1]^G$ 
and the products in \eqref{eq:s_chi} with at most two factors of the form $s_\chi$,  
and with at most $4$ factors in total. All these polynomials have degree at most $6$. 

The case $i=2$ is essentially the same. Indeed, 
the map $x_1\mapsto y_1$, $x_2\mapsto y_2$ extends to a $\field$-algebra  isomorphism between $\field[W_1]$ and $\field[W_2]$, which is also a $\mathrm{Dic}_8$-module isomorphism. The subgroup $\langle c\rangle$ acts trivially on $\field[W_1]$, and acts by multiplication by $-1$ of the variables $y_1,y_2$ in $\field[W_2]$. 
Therefore replacing $x_1$ by $y_1$ and $x_2$ by $y_2$ in the corresponding formulae 
above, we get 
the generators of 
$\field[W_2]^G$, 
as well as a direct sum decomposition into minimal $G$-invariant subspaces of a $\field G$-module direct complement in $\field[W_2]$ of the Hilbert ideal. 
So in the same way as in the above paragraph, we get the conclusion  
$\beta(W_2\oplus U)=6$.
\end{proof}

\begin{proposition} \label{prop:Q8xC2,V1+V2} 
We have $\beta(G,W_1\oplus W_2)=6$. 
\end{proposition}    

\begin{proof} 
The element  $a^2\in G$ multiplies each of the indeterminates $x_1,x_2,y_1,y_2$ by $-1$. 
Since $G$-invariants are $a^2$-invariant, they must involve only monomials of even degree. 
Thus  $\beta(G,W_1\oplus W_2)$ is even. On the other hand, $\beta^\field(G)=7$ by 
\cite[Proposition 3.2]{cziszter-domokos-szollosi}. We conclude that 
$\beta(G,W_1\oplus W_2)\le 6$. 
Note also that $\beta(G,W_1\oplus W_2)\ge 
\beta(G,W_1)$, and we saw in the course of the proof of Proposition~\ref{prop:Vi+U} 
that $\beta(G,W_1)=6$ (the third generator in 
\eqref{eq:Dic8 mingen} is not symmetric in $x_1$ and $x_2$, hence can not be expressed by the first two). 
\end{proof}    

\begin{proposition}\label{prop:ralative invariants Q8xC2} 
Let $v:=(w_1,w_2)\in W_1\oplus W_2$ with $w_1\neq 0$ and $w_2\neq 0$. 
Then for any $\chi\in \widehat G$ there exists 
a relative invariant $f$ of weight $\chi$ such that $\deg(f)\le 4$ and $f(v)\neq 0$. 
\end{proposition} 

\begin{proof} 
\[\begin{array}{c|c}
\chi & \text{relative invariants} \\ \hline \hline
(-1,1,1) & f_1:=x_1^2+x_2^2, \quad f_2:=x_1x_2(x_1^2-x_2^2) \\ \hline 
(-1,1,-1) & f_1:=x_1y_1+x_2y_2,\quad f_2:=x_1x_2(x_1y_1-x_2y_2), 
\quad f_3:=x_2^3y_1-x_1^3y_2 
\\ \hline 
(1,1,-1) & f_1:=x_1y_2-x_2y_1,\quad f_2:=x_1x_2(x_1y_2+x_2y_1), \quad f_3:=x_2^3y_2+x_1^3y_1
\end{array}\]  
A row in the above table contains a character $\chi\in \widehat G$ and relative invariants 
of weight $\chi$ of degree at most $4$, such that the common zero-set of these 
relative invariants is contained in the union of $W_1\oplus \{0\}$ and $W_2\oplus \{0\}$. 
For the other non-trivial weights the result can be deduced without further computation using 
Lemma~\ref{lemma:auto}. Indeed, 
the automorphism group of $G$ contains the automorphism group of 
its subgroup $\mathrm{Dic}_8$ as a subgroup. The automorphism group  of $\mathrm{Dic}_8$ acts 
(on the right) therefore on $\widehat G$ (an automorphism $\alpha$ sends $\chi$ to $\chi\circ\alpha$), 
and any non-trivial $\chi \in \widehat G$ is in the orbit of one of the three weights in the above table. 
Moreover, for an automorphism $\alpha$ of $\mathrm{Dic}_8$ (viewed as an automorphism of $G$) 
we have that $\psi_1\circ \alpha\cong \psi_1$ and $\psi_2\circ\alpha\cong \psi_2$.  
Observe that for $(w_1,w_2)\in W_1\oplus W_2$, the condition that none of 
$w_i$ ($i=1,2)$ is zero is equivalent to the condition that $\{g\cdot (w_1,w_2)\mid g\in G\}$ spans 
$W_1\oplus W_2$ as a $\field$-vector space. 
Therefore by Lemma~\ref{lemma:auto} (iii), no  $(w_1,w_2)$ with $w_1\neq 0$, $w_2\neq 0$ is contained 
in the common zero locus 
$\mathcal{V}(f\mid f\in \field[W_1\oplus W_2]^{G,\chi\circ\alpha}, 
\ \deg(f)\le 4)$ if no such element of $W_1\oplus W_2$ is contained in  
$\mathcal{V}(f\mid f\in \field[W_1\oplus W_2]^{G,\chi}, 
\ \deg(f)\le 4)$. 

Note finally that for the trivial weight $(1,1,1)\in \widehat G$ we may take $f=1\in \field[W_1\oplus W_2]^G$. 
\end{proof} 

\begin{theorem}\label{thm:sepbeta(Dic8xC2)} 
Assume that $\field$ contains an element of multiplicative order $8$.  
Then we have $\sepbeta^\field(\mathrm{Dic}_8\times \mathrm{C}_2)=6$. 
\end{theorem}

\begin{proof} 
By Theorem~\ref{thm:sepbeta index two} 
we know that $\sepbeta^\field(\mathrm{Dic}_8)=6$, and since $\mathrm{Dic}_8$ is a direct factor of $G$, 
we have $\sepbeta^\field(G)\ge \sepbeta^\field(\mathrm{Dic}_8)=6$. 

Now we turn to the reverse inequality. 
By Lemma~\ref{lemma:spanning invariants} (iii) it is sufficient to deal with the case when 
$|\field|$ is large enough so that we can apply Lemma~\ref{lemma:multfree}. 
Therefore it remains to prove that $\sepbeta(G,V)\le 6$ where $V=W_1\oplus W_2\oplus U$. 
That is, we have to show that if 
all invariants of degree at most $6$ take the same value on 
$v,v'\in V$, then $G\cdot v=G\cdot v'$. So 
$v=(w_1,w_2,u_{\chi}\mid \chi\in \widehat G)$ and 
$v'=(w'_1,w'_2,u'_\chi \mid \chi\in \widehat G)$.  
It follows from Proposition~\ref{prop:Q8xC2,V1+V2} that replacing $v'$ by an 
appropriate element in its $G$-orbit, we may assume that $w_1=w'_1$ and $w_2=w'_2$. 
If $w_1=0$ or $w_2=0$, then $v$ and $v'$ belong to the same orbit by Proposition~\ref{prop:Vi+U}. 
For any $\chi\in \widehat G$ there exists a relative invariant 
$f\in \field[W_1\oplus W_2]^{G,\chi^{-1}}$ with $\deg(f)\le 4$ and $f(w_1,w_2)\neq 0$ 
by Proposition~\ref{prop:ralative invariants Q8xC2}. 
Consider the invariant $ft_{\chi}\in \field[V]^G$: it has degree at most $5$, and so from 
$(ft_{\chi})(v)=(ft_{\chi})(v')$ we deduce $u_\chi=u'_\chi$. This holds for all 
$\chi\in\widehat G$, 
thus we showed $v=v'$, as claimed. 
\end{proof}

\subsection{The group $\mathrm{C}_4\rtimes \mathrm{C}_4$} 

In this section 
\[G=\mathrm{C}_4\rtimes \mathrm{C}_4=\langle a,b\mid a^4=1=b^4,\quad 
bab^{-1}=a^3\rangle.\] 
Assume that $\field$ has an element $\mathrm{i}$ of multiplicative order $4$. 
Then $G$ has two irreducible $2$-dimensional representations (up to isomorphism), 
namely 
\[\psi_1:a\mapsto 
\begin{bmatrix} 
      \mathrm{i}& 0 \\
      0 & -\mathrm{i}\\
   \end{bmatrix}, \quad 
   b\mapsto \begin{bmatrix} 
      0 & 1 \\
      1 & 0 \\
   \end{bmatrix}, \qquad 
   \psi_2: a\mapsto 
\begin{bmatrix} 
      \mathrm{i}& 0 \\
      0 & -\mathrm{i}\\
   \end{bmatrix}, \quad 
   b\mapsto \begin{bmatrix} 
      0 & -1 \\
      1 & 0 \\
   \end{bmatrix}.\] 
  The other irreducible representations of $G$ are $1$-dimensional, and can be labeled by 
  the elements of the subgroup $\widehat G:=\{\pm1\}\times \{\pm1,\pm\mathrm{i}\}$ of $\field^\times\times \field^\times$ as follows: 
  identify  $\chi=(\chi_1,\chi_2)\in \widehat G$ with the representation 
  $\chi:G\to \field^\times$ given by  
  \[\chi:a\mapsto \chi_1,\quad b\mapsto \chi_2.\] 
 Denote by $W_j$ the vector space $\field^2$ endowed with the  representation $\psi_j$ for $j=1,2$, and for $\chi\in \widehat G$ denote by $U_\chi$ the vector space $\field$ endowed with the  representation $\chi$. 
 Set $U:=\bigoplus_{\chi\in \widehat G}U_\chi$. 
 Note that none of $\psi_1$ or $\psi_2$ is faithful: $\ker(\psi_1)=\langle b^2\rangle$, 
 $G/\ker(\psi_1)\cong \mathrm{D}_8$, and $\psi_1$ is the lift to $G$ of the  only $2$-dimensional representation 
 of $\mathrm{D}_8$. On the other hand, $\ker(\psi_2)=\langle a^2b^2\rangle$, 
 $G/\ker(\psi_2)\cong \mathrm{Dic}_8$, and $\psi_2$ is the lift to $G$ of the only $2$-dimensional irreducible 
 representation of the quaternion group. 
The coordinate functions on $W_1$ are denoted by $x_1,x_2$, the coordinate functions 
 on $W_2$ are $y_1,y_2$, and the coordinate function on $U_\chi$ is $t_\chi$ for $\chi\in \widehat G$. 
 It is well known and easy to see that 
 \begin{align*}\field[x_1,x_2]^G&=\field[x_1x_2,x_1^4+x_2^4],
 \\ \field[y_1,y_2]^G&=\field[(y_1y_2)^2,\ y_1^4+y_2^4,\ y_1y_2(y_1^4-y_2^4)]. 
 \end{align*}
 Consider the following relative invariants (their index indicates their weight): 
 \begin{align*}r_{(-1,1)}:=x_1^2+x_2^2, \quad r_{(-1,-1)}:=x_1^2-x_2^2, \quad 
 r_{(1,-1)}:=r_{(-1,1)}r_{(-1,-1)}=x_1^4-x_2^4
 \\ 
s_{(1,-1)}:=y_1y_2,\quad s_{(-1,1)}:=y_1^2+y_2^2, \quad s_{(-1,-1)}:=y_1^2-y_2^2.
\end{align*} 
Using the above observations, one can prove by arguments 
similar to the proof of Proposition~\ref{prop:Vi+U} the following: 

  \begin{proposition}\label{prop:V1+U} 
 For $i=1,2$ we have 
 $\beta(G,W_i\oplus  U)=6$.  
 \end{proposition}  
 
\begin{proposition} \label{prop:V1+V2}
We have $\beta(G,W_1\oplus W_2)\le 6$. 
\end{proposition} 
\begin{proof} 
We know that $\beta^\field(G)\le 7$ by \cite[Proposition 3.1]{cziszter-domokos-szollosi}. 
On the other hand, the element $a^2\in G$ acts on $W_1\oplus W_2$ as 
multiplication by the scalar $-1$, hence any homogeneous element of 
$\field[W_1\oplus W_2]^G$ has even degree. It follows that $\beta(G,W_1\oplus W_2)$ 
is an even number, less than or equal to $7$. Consequently, 
$\beta(G,W_1\oplus W_2)\le 6$. 
\end{proof} 

\begin{proposition}\label{prop:V1+V2+U(1,i)} 
We have $\sepbeta(W_1\oplus W_2\oplus U_{(1,\mathrm{i})})\le 6$. 
\end{proposition} 

\begin{proof} 
Let $v=(w_1,w_2,u)\in V$ and $v'=(w'_1,w'_2,u')\in V$ 
such that $f(v)=f(v')$ for all $f\in \field[V]^G$ with $\deg(f)\le 6$ 
(here $w_i,\ w'_i\in W_i$ and $u,u'\in U$). We need to show that $v$ and $v'$ 
have the same $G$-orbit. By Proposition~\ref{prop:V1+V2}, $(w_1,w_2)$ and $(w'_1,w'_2)$ are in the same $G$-orbit. Replacing $v$ by an appropriate element in its $G$-orbit, we may assume 
that $w_1=w'_1$ and $w_2=w'_2$. If $w_1=0$ or $w_2=0$, then we have $G\cdot v=G\cdot v'$  by Proposition~\ref{prop:V1+U}. 
From now on we assume that none of $w_1$ or $w_2$ is zero. 
Set 
\[r^{(1)}_{(1,-\mathrm{i})}:=x_1y_2+\mathrm{i}x_2y_1, \quad 
r^{(2)}_{(1,-\mathrm{i})}:=x_1y_1^3-\mathrm{i}x_2y_2^3, \quad 
r_{(1,\mathrm{i})}:=x_1y_2-\mathrm{i}x_2y_1.\] 
Then $r^{(1)}_{(1,-\mathrm{i})}$, $r^{(2)}_{(1,-\mathrm{i})}$, $t^3$ are relative invariants of weight 
$(1,-\mathrm{i})\in \widehat G$, whereas $r_{(1,\mathrm{i})}$, $t$ are relative invariants of weight 
$(1,\mathrm{i})$. 
Therefore we have the invariants 
\[f_1:=t^4,\quad f_2:=r^{(1)}_{(1,-\mathrm{i})}t,\quad f_3:=r^{(2)}_{(1,-\mathrm{i})}t,\quad 
f_4:=r_{(1,\mathrm{i})}t^3.\] 
All of the above invariants have degree at most $5$. 
By assumption, $f_j(v)=f_j(v')$ for $j=1,2,3,4$. 
The equality $t^4(v)=t^4(v')$ implies that $u'\in \{u,-u,\mathrm{i}u,-\mathrm{i}u\}$. 
In particular, if $u=0$, then $u'=0$, and hence $v=v'$. 
It remains to deal with the case $u\neq 0$. 
Moreover, $u=u'$ if and only if $t(v)=t(v')$ if and only if $t^3(v)=t^3(v')$. 
Suppose for contradiction that $u\neq u'$. Then $f_j(v)=f_j(v')$ for $j=2,3,4$ and 
$(w_1,w_2)=(w'_1,w'_2)$ imply  
that 
\[0=r^{(1)}_{(1,-\mathrm{i})}(v)=r^{(2)}_{(1,-\mathrm{i})}(v)=r_{(1,\mathrm{i})}(v).\]
 From $0=r^{(1)}_{(1,-\mathrm{i})}(v)=r_{(1,\mathrm{i})}(v)$ we conclude 
 $0=x_1(v)y_2(v)=x_2(v)y_1(v)$. Taking into account that $w_1\neq 0$ and $w_2\neq 0$, 
 we conclude that $0=x_1(v)=y_1(v)$ or $0=x_2(v)=y_2(v)$. 
 Then $r^{(2)}_{(1,-\mathrm{i})}(v)=0$ yields  that both $x_1(v)y_1(v)$ and $x_2(v)y_2(v)$ are zero. 
 We deduce $w_1=0$ or $w_2=0$, a contradiction. 
\end{proof} 

\begin{proposition} \label{prop:V1+V2+U(-1,1)} 
We have $\sepbeta(G,W_1\oplus W_2\oplus U_{(-1,1)})\le 6$ 
and $\sepbeta(G,W_1\oplus W_2\oplus U_{(1,-1)})\le 6$. 
\end{proposition} 

\begin{proof} 
First we deal with $V:=W_1\oplus W_2\oplus U_{(-1,1)}$. 
Let $v=(w_1,w_2,u)\in V$ and $v'=(w'_1,w'_2,u')\in V$ 
such that $f(v)=f(v')$ for all $f\in \field[V]^G$ with $\deg(f)\le 6$. We need to show that 
$v$ and $v'$ have the same $G$-orbit.
In the same way as in the proof of Proposition~\ref{prop:V1+V2+U(1,i)}, 
we may assume that $w_1=w'_1$, $w_2=w'_2$ and none of $w_1$ or $w_2$ is zero. 
It follows that 
\[r_{(-1,1)}=x_1^2+x_2^2,\quad 
r_{(-1,-1)}s_{(1,-1)}=(x_1^2-x_2^2)y_1y_2, 
\quad s_{(-1,1)}=y_1^2+y_2^2\]
can not all vanish at $v$. Suppose that $f$ is one of the above polynomials, 
with $f(v)\neq 0$. 
Since $f$ is a relative invariant of weight $(-1,1)$ of degree at most $4$, 
we get that $ft\in \field[V]^G$ has degree at most $5$. Moreover, the equality 
$(ft)(v)=(ft)(v')$ implies $u=u'$, hence $v=v'$. 

The proof for $V=W_1\oplus W_2\oplus U_{(1,-1)}$ is similar, the only difference is that we use the relative invariants 
\[s_{(1,-1)}=y_1y_2, \quad 
s_{(-1,1)}r_{(-1,-1)}=(y_1^2+y_2^2)(x_1^2-x_2^2), 
\quad s_{(-1,-1)}r_{(-1,1)}=(y_1^2-y_2^2)(x_1^2+x_2^2).\] 
\end{proof}

\begin{proposition}\label{prop:mu(C4rtimesC4)} 
We have $\mu(G)\le 2$, and  hence $\kappa(G)\le 3$. 
\end{proposition} 
\begin{proof} 
Let $S$ be an intersection independent set of subgroups of $G$ with $|S|=\mu(G)$. 
Clearly $G\notin S$. If $S$ contains a cyclic subgroup of prime power order, 
then $|S|\le 2$ by Lemma~\ref{lemma:prime power}. From now on assume that 
$S$ contains no cyclic subgroup of prime power order. 
All such proper subgroups of $G$ contain 
the subgroup $\langle a^2,b^2\rangle$. If $\langle a^2,b^2\rangle\in S$, then $|S|=1$. 
Otherwise $S$ consists of two order $8$ subgroups of $G$. In either case, $|S|\le 2$. 
The inequality $\mu(G)\le 2$ is proved, and hence 
$\kappa(G)\le 1+\mu(G)\le 3$ holds by \eqref{eq:kappa-mu}. 
\end{proof}

\begin{theorem}\label{thm:sepbeta(C4rtimesC4)} 
Assume that $\field$ has an element of multiplicative order $8$. 
Then we have $\sepbeta(\mathrm{C}_4\rtimes \mathrm{C}_4)=6$. 
\end{theorem} 

\begin{proof} 
The factor group of $G$ modulo its central normal subgroup $\langle a^2b^2\rangle$ 
is isomorphic to the quaternion group $\mathrm{Dic}_8$ of order $8$. Therefore $\sepbeta^\field(G)\ge \sepbeta^\field(\mathrm{Dic}_8)=6$ 
(in fact we have $\sepbeta(G,W_2)=6$, as we saw in Section~\ref{sec:indextwo}). 

To prove the reverse inequality, by Lemma~\ref{lemma:spanning invariants} we may assume that $\field$ is large enough so that 
Lemma~\ref{lemma:multfree} applies and we can reduce to the study of multiplicity-free representations. 
By Proposition~\ref{prop:mu(C4rtimesC4)} we have $\kappa(G)\le 3$, and therefore 
by Lemma~\ref{lemma:helly} it is sufficient to show that $\sepbeta(V)\le 6$ if $V$ is the direct sum of three pairwise non-isomorphic $\field G$-modules. Assume that $V$ is such a $\field G$-module. 
By Proposition~\ref{prop:V1+V2},  Proposition~\ref{prop:V1+U} we know that $\sepbeta(G,V)\le \beta(G,V)\le 6$, unless 
$V=W_1\oplus W_2\oplus U_{\chi}$ for some non-trivial character
$\chi \in \widehat G$. The cases when $\chi \in \{(1,\mathrm{i}),\ (-1,1),\ (1,-1)\}$ are covered by Proposition~\ref{prop:V1+V2+U(1,i)}, Proposition~\ref{prop:V1+V2+U(-1,1)}. 
 The remaining representations are of the form $\rho\circ \alpha$, where $\rho$ is one of the above three representations, and $\alpha$ is an automorphism of $G$. So the proof is complete by Lemma~\ref{lemma:auto}.  
\end{proof} 

\section{Comments on computational aspects}\label{sec:algorithm}

It is not possible to obtain the results of this paper on small groups  
by computer for two reasons. They are valid in infinitely many different characteristics 
and there are no known general principles for relating the separating 
Noether numbers of a group over fields of different characteristic 
(unlike for the Noether number, for which some interesting results in this direction 
can be found in \cite[Theorem 4.7]{knop}). 
Moreover, our results are proved also 
over fields that are not algebraically closed. 
While by Lemma~\ref{lemma:spanning invariants} (iii) one can 
get an upper bound on the separating Noether number 
by passing to the algebraic closure of the base field, 
there seems to be no algorithm for finding the exact value of the separating Noether number 
over an infinite algebraically non-closed base field 
(in other words, our arguments giving lower bound for separating Noether numbers 
can not be replaced by computer calculations). 
Note that for algebraically closed base fields and not necessarily finite 
reductive groups, an algorithm to compute a separating set of invariants is given in 
\cite{kemper:algorithm} (without a prior computation of a generating set of invariants). 

However, one can in principle verify by computer for chosen base fields that our 
claimed value is at least an upper bound for the separating Noether number 
as follows. We can pass to the algebraic closure of the base field, and 
by Lemma~\ref{lemma:spanning invariants} (iii) the separating Noether number 
for the algebraic closure is an upper bound for the separating Noether number 
over the original base field. 
There are computer algebra packages that can compute 
minimal systems of generators for rings of invariants of finite groups. 
Using a minimal homogeneous generating system of $\field[V]^G$ as input, 
assuming that $\field$ is algebraically closed, 
a simple algorithm computes $\sepbeta(G,V)$, based on the 
following statement. 

\begin{proposition}\label{prop:nullstellensatz} 
Let $V$ be a $\field G$-module, where $\field$ is algebraically closed. 
Let $\mathcal{A}$ be a minimal homogeneous generating system 
of $\field[V]^G=\field[x_1,\dots,x_n]^G$. 
Take additional variables $y_1,\dots,y_n$, and denote by 
$\delta:\field[x_1,\dots,x_n]\to \field[x_1,\dots,x_n,y_1,\dots,y_n]$
the map $f\mapsto f(x_1,\dots,x_n)-f(y_1,\dots,y_n)$. 
For a positive integer $d$ we have that 
$\sepbeta(G,V)\le d$ if and only if for all $f\in \mathcal{A}$ 
with $\deg(f)>d$ we have that $\delta(f)$ is contained in the radical of 
the ideal $(\delta(h)\mid h\in \mathcal{A},\ \deg(h)\le d)$. 
Consequently,
\[\sepbeta(G,V)=\min\{d\mid \forall f
\in \mathcal{A}\colon 
\delta(f) \in \sqrt{(\delta(h)\mid h\in \mathcal{A}, \ \deg(h)\le d)}\}.\]
\end{proposition}

\begin{proof}
This is a straightforward consequence of
Definition~\ref{def:separating set}, the definition of $\sepbeta(G,V)$, 
and Hilbert's Nullstellensatz. 
\end{proof}

We have implemented in SageMath an algorithm based on Proposition~\ref{prop:nullstellensatz}. 
In the case $\field=\mathbb{C}$, we verified the values of $\sepbeta^\field(G)$ for the non-abelian groups given in Theorem~\ref{thm:sepbeta(<32)} 
using the online CoCalc platform \cite{CoCalc} and some good apriori upper bounds on 
the Helly dimension $\kappa(G)$ or $\kappa(G,V)$ (cf. Section~\ref{sec:helly}). 



\section*{Declarations} 

\noindent{\bf{Ethical Approval:}} not applicable

\noindent{\bf{Funding:}} This research was supported by the Hungarian National Research, Development and Innovation Office,  NKFIH K 138828.

\noindent{\bf{Availability of data and materials:}} not applicable


\end{document}